\documentclass[11pt, reqno]{amsart}

\usepackage[utf8]{inputenc}
\usepackage{graphics,graphicx}
\usepackage{pstricks,pst-node,pst-tree}
\usepackage{tikz}
\usepackage{caption}
\usepackage{subcaption}
\usepackage{mathtools,mathrsfs}
\usepackage{amsmath, amssymb, amsfonts,amsthm}
\usepackage{textcomp}
\usepackage{newlfont}

\usepackage{hyperref}

\usetikzlibrary {positioning}
\definecolor {processblue}{cmyk}{1,0.5,0,0}

\newcommand{\cp}{\ \stackrel{P}{\longrightarrow} \ }

\newcommand{\FF}{\mbox{${\mathcal F}$}}

\newcommand{\Pbold}{\mbox{${\mathbb P}$}}

\newcommand{\bE}{{\mathbf E}}

\newcommand{\bone}{\mathbf{1}}

\newtheorem{theorem}{Theorem}

\newtheorem{proposition}{Proposition}

\theoremstyle{remark}
\newtheorem{remark}{Remark}

\theoremstyle{definition}

\newtheorem{definition}{Definition}

\begin{document}
\title[Urn Schemes with Negative but Linear Reinforcements]{Generalized P\'olya Urn Schemes with Negative but Linear Reinforcements}
\author{Antar Bandyopadhyay}
\address[Antar Bandyopadhyay]{Theoretical Statistics and Mathematics Unit \\
         Indian Statistical Institute, Delhi Centre \\ 
         7 S. J. S. Sansanwal Marg \\
         New Delhi 110016 \\
         INDIA}
\address{Theoretical Statistics and Mathematics Unit, 
         Indian Statistical Institute, Kolkata;
         203 B. T. Road, Kolkata 700108, INDIA}
\email{\href{mailto:antar@isid.ac.in}{antar@isid.ac.in}}       
\author{Gursharn Kaur}  
\address[Gursharn Kaur]{Theoretical Statistics and Mathematics Unit \\
         Indian Statistical Institute, Delhi Centre \\ 
         7 S. J. S. Sansanwal Marg \\
         New Delhi 110016 \\
         INDIA}
\email{\href{mailto:gursharn.kaur24@gamil.com}{gursharn.kaur24@gmail.com}}

\date{\today}

\selectfont

\begin{abstract}
In this paper, we consider a new type of urn scheme, where the selection 
probabilities are proportional to a weight function, which is linear but decreasing
in the proportion of existing colours. 
We refer to it as the \emph{negatively reinforced} urn scheme. 
We establish almost sure limit of the random configuration for any \emph{balanced}
replacement matrix $R$. In particular, we show that the limiting configuration is 
uniform on the set of colours, if and only if, $R$ is a \emph{doubly stochastic} matrix.
We further establish almost sure limit of the vector of colour counts and 
prove central limit theorems for the random configuration, as well as, for the colour 
counts.	
\end{abstract}



\maketitle

\section{Introduction}

\subsection{Background and Motivation}
\label{SubSec:Background}
Various kinds of \emph{random reinforcement models} have been of much interest in recent years
\cite{Da90, Vl03, Svante2, BaiHu05, FlDuPu06, VlTa07, Pe07, CoVl09, LaPa13, CoCoLi13, ChHsYa14}. 
\emph{Urn schemes}, which were first studied by P\'{o}lya \cite{Polya30}, are perhaps the simplest 
reinforcement models. They have many applications and generalizations 
\cite {Fri49, Free65, AthKar68, BagPal85, Pe90, Gouet, Svante1, Svante2, BaiHu05, maulik1, maulik2, maulik3,
CrGeNiVol11, LaPa13, ChHsYa14, BaTh14a, BaTh14b,  BaTh14c}.
In general, reinforcement models typically adhere to the structure of ``\emph{rich get richer}'', which can also be termed as 
\emph{positive reinforcement}. 
However, there have been some studies on \emph{negative reinforcements} models in the context of percolation, such as the
\emph{forest fire}-type models from the point-of-view of \emph{self-destruction} 
\cite{vanBr04, RaToth09, CraFreeToth15, AhlSidTyk14, AhlDumKozSid15} and 
\emph{frozen percolation}-type models from the point-of-view of \emph{stagnation} \cite{Al00, Ba06, vanKissNo12, vandeLimaNo12, vanNo17}. 
For urn schemes, a type of ``negative reinforcement'' have been studied when balls can be thrown away from the urn, as well as,
added \cite{EhEh90, WiMaIl98, King99, KingVol03,CrGeNiVol11}. 
In such models, it is usually assumed that the model is \emph{tenable}, that is,
regardless  of  the  stochastic  path  taken  by  the
process, it is never required to remove a ball of a colour not currently present in the
urn. Perhaps the most famous of such scheme is the \emph{Ehrenfest urn} \cite{EhEh90, Mah09}, which models the diffusion of a gas between
two chambers of a box. There are some models without tenability, such as the \emph{OK Corral Model} \cite{WiMaIl98, King99, KingVol03} or 
\emph{Simple Harmonic Urn} \cite{CrGeNiVol11} in two colors. Typically these are used for modeling \emph{destructive competition}.

In recent days, there has been some work on \emph{negative reinforcements},
random graphs  \cite{SevRik2006, SevRik2008, BaSe14} from a different point-of-view, where 
attachment probabilities of a new vertex are decreasing functions of the degree of the existing vertices.
Such models have also been referred as ``\emph{de-preferential attachment}'' \cite{BaSe14} as opposed to usual  
``preferential'' attachment models \cite{AlBara99}.
Motivated by this later set of works, in this paper, we present a specific model of \emph{negatively reinforced urn scheme}, where
the selection probabilities are linear but decreasing function of the proportion of colours.
Negatively reinforced urn schemes are natural models for modeling problems with resource constrains. 
In particular, multi-server queuing systems with capacity constrains \cite{LucMcDiar2005, LucMcDiar2006}. 
For such cases, it is desirable that at the steady state limit, all agents are having equal loads. 
In this work, we show that for a negative but linearly reinforced urn scheme 
such a limit is obtained under fairly general conditions on the replacement mechanism. 
%
%

\subsection{Model Description}
\label{SubSec:Model}
In this work, we will only consider \emph{balanced} urn schemes with $k$-colours, index by $S := \left\{0,1,\ldots, k-1\,\right\}$.  
More precisely, if $R := \left(\left(R_{i,j}\right)\right)_{0 \leq i,j \leq k-1}$ denotes the \emph{replacement matrix}, 
that is, 
$R_{i,j} \geq 0$ is the \emph{number of balls of colour $j$ to be placed in the the urn when the colour of the selected ball 
is $i$}, then
for a balanced urn, all row sums of $R$ are constant. In this case, 
dividing all entries by the common row total, 
we may assume $R$ is a \emph{stochastic matrix}. 
We will also assume that the starting configuration $U_0:=\left(U_{0,j}\right)_{0 \leq j \leq k-1}$ 
is a probability distribution on the set of colours $S$. As we will see from the proofs of our main results, 
this apparent loss of generality can easily be removed. 

Denote by $U_n := \left(U_{n,j}\right)_{0 \leq j \leq k-1} \in [0, \infty)^k$
the random configuration of the urn at time $n$. Also let $\FF_n := \sigma\left(U_{0}, U_{1}, \cdots, U_{n}\right)$ be the
natural filtration. We define a random variable $Z_n$ by
\begin{equation}
\Pbold\left( Z_n = j \,\Big\vert\, \FF_n \right) 
\propto 
w_\theta\left(\frac{U_{n,j}}{n+1}\right), \,\,\, 0 \leq j \leq k-1.
\label{Equ:Def-Z_n}
\end{equation}
where $w_\theta :[0,1] \to \mathbb{R}_+$ is given by
\begin{equation}
\label{weight-function}
w_\theta(x) =\theta -x.
\end{equation}
$\theta \geq 1$ will be considered as a parameter for the model.  
Note that, $Z_n$ represents the colour chosen at the $\left(n+1\right)$-th draw. 
Starting with $U_0$ we define $\left(U_n\right)_{n \geq 0}$ recursively as follows:
\begin{equation}
U_{n+1} = U_n + \chi_{n+1} R.
\label{Equ:Basic-Recursion} 
\end{equation}
where $\chi_{n+1} := \left(\bone\left(Z_n = j \right)\right)_{0 \leq j \leq k-1}$. 

We call the process $\left(U_n\right)_{n \geq 0}$, a \emph{negative but linearly reinforced urn scheme} with
initial configuration $U_0$ and replacement matrix $R$. 
In this work, we study the asymptotic properties of the following two processes: \\

\noindent
{\bf Random configuration of the urn:}
Observe that for all $n \geq 0$,
\begin{equation}
\sum_{j=0}^{k-1} U_{n,j} = n + 1.
\label{Equ:Sum-of-weights}
\end{equation}
This holds because $R$ is a stochastic matrix and $U_0$ is a probability vector. 
Thus the \emph{random configuration of the urn}, namely, $\displaystyle{\frac{U_n}{n+1}}$
is a probability mass function. Further,  
\begin{equation}
\Pbold\left(Z_n = j \,\Big\vert\, \FF_n \right) 
= \frac{\theta}{k \theta - 1} - \frac{1}{k \theta -1} \frac{U_{n,j}}{n+1}
, \,\,\, 0 \leq j \leq k-1. 
\label{prob}
\end{equation}
Thus, $\displaystyle{\frac{U_n A}{n+1}}$ is the conditional distribution of the $(n+1)$-th
selected colour, namely $Z_n$, given $U_0, U_1, \ldots, U_n$, where
\begin{equation}
A_{k \times k} = \frac{\theta}{k \theta -1} J_k - \frac{1}{k \theta -1} I_k,
\label{Equ:Def-A}
\end{equation}
and $J_{k} := \bone^{T}\bone$ is the $k \times k$ matrix with all entries equal to $1$ and $I_{k}$ is the 
$k \times k$-identity matrix.\\

\noindent
{\bf Color count statistics:}
Let $N_n := \left(N_{n,0}, \hdots , N_{n,k-1}\right)$ be the vector of length $k$,
whose $j$-th element is the number of times colour $j$ was selected in the first $n$ trials, that is
\begin{equation}
\label{Def:Nn}
N_{n,j} = \sum_{m=0}^{n-1} \bone\left(Z_m = j\right), \,\,\, 0 \leq j \leq k-1.
\end{equation}
It is easy to note that from  ~~\eqref{Equ:Basic-Recursion} it follows
\begin{equation}
\label{Un-Nn}
U_{n+1} = U_0 + N_{n+1} R. 
\end{equation}

\subsection{Outline}
\label{SubSec:Outline}
In Section \ref{Sec:Main-Results} we present the main results of the paper and the proofs are given in 
Section \ref{Sec:Proof-of-Necessary-Sufficient-Condition-on-hatR} and
Section \ref{Sec:Proofs}. 

\section{The Main Results}
\label{Sec:Main-Results}

We define a new $k \times k$ stochastic matrix, namely
\begin{equation}
\label{Def:R-hat}
\hat{R} := RA = \frac{1}{k \theta -1} \left(\theta J_k - R\right),
\end{equation}
where $A$ is as defined in ~~\eqref{Equ:Def-A}. As we state in the sequel, the asymptotic properties of 
$\left(U_n\right)_{n \geq 0}$ and $\left(N_n\right)_{n \geq 0}$ depends on whether the stochastic matrix
$\hat{R}$, is \emph{irreducible} or \emph{reducible}. We first state a necessary and sufficient condition for that. 

\subsection{A Necessary and Sufficient Condition for $\hat{R}$ to be Irreducible}
\label{SubSec:Star-Condition}
We start with the following definitions, which are needed for stating our main results.
\begin{definition}
 A directed graph $\mathcal{G} =(\mathcal{V}, \vec{\mathcal{E}} )$ is called the graph associated with a $k\times k$ 
 stochastic $R = ((R_{i,j}))_{0\leq i,j\leq k-1}$, if 
 \[\mathcal{V} =  \{ 0, 1,\hdots, k-1\}\;\;  \text{and}\;\; \vec{\mathcal{E}} =\{(\overrightarrow{i,j})| R_{i,j}>0; i,j\in \mathcal{V}\}.\]
\end{definition}

\begin{definition}
A stochastic matrix $R$ is called a \emph{star}, if there exists a $j \in \{0,1,\cdots, k-1\}$,
such that,
\[R_{i,j} = 1 \;\; \text{ for all  } i\neq j,\] 
and in that case, we say $j$ is the central vertex.
\end{definition}

By definition, for the graph associated with a star replacement matrix, 
there is a central vertex such that each vertex other than the central vertex has only one outgoing edge and 
that is towards the central vertex. We note that in the definition of a star we allow the central vertex to have a self loop.
%

As we will see in the sequel, the asymptotic properties will
depend on the irreducibility of the  (new) stochastic matrix $\hat{R}$, as defined in 
~\eqref{Def:R-hat}. Following lemma provides a necessary and sufficient condition for $\hat{R}$ to be irreducible.

\begin{proposition}\label{Main-Proposition}
Let $R$ be a $k \times k$ stochastic matrix with $k \geq 2$, then $\hat{R}$ is irreducible, 
if and only if either $\theta >1$ or $\theta =1 $ but $R$ is not a star.
\end{proposition}

\subsection{Asymptotics of the Random Configuration of the Urn}
\label{SubSec:Random-Config-Results}

\subsubsection{Case when $\hat{R}$ is Irreducible.}
Our first result is the almost sure asymptotic of the colour proportions. 
\label{SubSubSec:Random-Config-Results-Irreducible}
\begin{theorem} 
\label{Thm-Un-Strong-convergence}
Let $\hat{R}$ be irreducible. Then, for every starting configuration $U_0$,
\begin{equation}
\frac{U_{n,j}}{n+1} \longrightarrow \mu_j,\;\; a.s.\;\;\;\;\forall \;0\leq j\leq k-1,
\label{Equ:Limit-mu}
\end{equation}
where  $\mu = \left(\mu_0,\mu_1,\cdots, \mu_{k-1}\right)$ is the unique solution of the following matrix equation
\begin{equation}
\left(\theta \bone - \mu\right) R = \left(k \theta -1\right) \mu.
\label{Equ:Def-mu}
\end{equation}
\end{theorem}

\begin{remark}
Notice that if we define $\nu = \mu A$, then from the equations ~\eqref{Equ:Def-A} and ~\eqref{Equ:Def-mu}, it follows that 
$\nu$ is the unique solution of the matrix equation 
$\nu \hat{R} = \nu$. Further, from equation ~\eqref{Equ:Def-mu} we have $\mu = \nu R$.
\end{remark}

\begin{remark} 
 Since $\frac{U_{n,j}}{n+1}$ is a bounded random variable, thus we get
\begin{equation}
\frac{\bE\left[U_{n,j}\right]}{n+1} \longrightarrow \mu_j, \;\;a.s., \;\;\; \forall \;0\leq j\leq k-1,
 \end{equation}
 where $\mu$ satisfies equation ~\eqref{Equ:Def-mu}.
\end{remark}

\begin{remark}
It is worth to note here that, the stochastic matrices $R$ and $\hat{R}$ both have uniform distribution 
as their unique stationary distribution, if and only if, $R$ is doubly stochastic, that is when 
$\bone R = \bone$.

\end{remark}

Our next result is a \emph{central limit theorem} for the colour proportions. 
\begin{theorem}
\label{Thm-Un-CLT}
Suppose $\hat{R}$ is irreducible then there exists a $k\times k$ \emph{variance-co-variance} matrix $\Sigma \equiv \Sigma\left(\theta, k\right)$, 
such that,
\begin{equation}
 \frac{ U_n -n\mu}{\sigma_n}  \implies {\mathcal N}_k(0,\Sigma),
\end{equation}
where for $k \geq 3$, 
\begin{equation}\label{Equ:Def-sigma-n-k}
 \sigma_n  =\begin{cases}
            \sqrt{n \log n} & \text{if}  \;\; k = 3, \theta = 1 \text{\ and one of the eigenvalue of\ } R \text{ is} -1,\\
                            & \\
            \sqrt{n} & \text{otherwise}.
	  \end{cases}
\end{equation}
and for $k=2$ and $\theta \in \left[1, \frac{3}{2}\right], $
\begin{equation}\label{Equ:Def-sigma-n-2}
 \sigma_n  =\begin{cases}
            \sqrt{n \log n} & \text{if}  \;\; \text{the eigenvalues of\ } R \text{\ are\ } 1 
                                              \text{\ and\ } \lambda = \frac{1 - 2\theta}{2}; \\
                            & \\
            \sqrt{n}        & \text{if}  \;\; \text{the eigenvalues of\ } R \text{\ are\ } 1 
                                              \text{\ and\ } \lambda < \frac{1 - 2\theta}{2}. \\
	  \end{cases}
\end{equation}
\end{theorem}

\begin{remark}
Note that $\Sigma$ is necessarily a \emph{positive semi-definite} matrix because of ~\eqref{Equ:Sum-of-weights}.
\end{remark}

\begin{remark}
It is worth noting here that the scaling is always by $\sqrt{n}$ for any parameter value $\theta \geq 1$ when $k \geq 4$. 
However, for small number of colors, namely, $k \in \left\{2, 3\right\}$, and certain specific parameter values,  
as given in equation \eqref{Equ:Def-sigma-n-k} and  \eqref{Equ:Def-sigma-n-2}above has an extra factor of $\sqrt{\log n}$.
\end{remark}

\subsubsection{Case when $\hat{R}$ is Reducible.} 
\label{SubSubSec:Random-Config-Results-Reducible}
By Proposition \ref{Main-Proposition}, we know that $\hat{R}$ can be reducible, if and only if, $R$ is
star and $\theta =1$. Suppose $R$ is a star with $k\geq 2$ colours, then without any loss of generality we can write 
\begin{equation} \label{R-Star}
R = \begin{pmatrix}
\alpha_0&\alpha_1 & \dots & \alpha_{k-1}\\
1&0&\dots &0\\
\vdots&\vdots&\ddots &\vdots\\
1&0&\dots &0\\
\end{pmatrix}\;\;\;\text{with }\;\; \sum_{j=0}^{k-1} \alpha_j = 1, \; \text{and} \; \alpha_j\geq 0, \; \; \forall j,
\end{equation}
by taking $0$ as the central vertex. Taking $\theta = 1$, the matrix $\hat R$ is

\begin{equation} \label{R-hat-star}
\hat{R} = \frac{1}{k-1}\begin{pmatrix}
1-\alpha_0&1-\alpha_1 & \dots & 1-\alpha_{k-1}\\
0&1&\dots &1\\
\vdots&\vdots&\ddots &\vdots\\
0&1&\dots &1\\
\end{pmatrix},
\end{equation}
which is clearly reducible. In the next theorem, we describe the limit of the urn configuration.

\begin{theorem} \label{Thm-star-1}
Let $\theta =1$ and replacement matrix $R$ be a star matrix as given in equation ~\eqref{R-Star} 
and $R\neq \begin{bmatrix} 0&1\\1&0 \end{bmatrix}$ then, 

\begin{equation} \label{limit-Un-star}
\frac{U_{n,0}}{n+1} \longrightarrow 1, \;\; a.s. 
\end{equation}
Further, there exists a random variable $W \geq 0$, with $\bE\left[W\right]>0$, such that,
\begin{equation} \label{limit-Un-star-W}
\frac{U_{n,j}}{n^{\gamma}} \longrightarrow \frac{\alpha_j}{k-1}W, \;\;a.s.\;\;\forall j= 1,2,\hdots, k-1, 
\end{equation}
where $\gamma = \frac{1-\alpha_0}{k-1} <1$.
\end{theorem}
 
 \begin{remark}
  In a trivial case, when  $\gamma = 0$ or $(\alpha_0 = 1)$ we have
  \[ U_{n,0} = U_{0,0}+n\]
  and \[ U_{n,j} = U_{0,j} \;\;\;\;\text{ for all } \;\; j =1,2,\cdots, k-1.\]
  That is, at every time $n$, only colour$1$ is reinforced into the urn.
  \label{Rem:Gamma-0-Case}
 \end{remark}

 \begin{remark}
 \label{Friedman-2color-Un}
 When $R =  \begin{bmatrix} 0&1\\1&0  \end{bmatrix}$, we get $\hat{R} = \begin{bmatrix} 1&0\\ 0&1 \end{bmatrix}.$ 
 Notice that then $\hat{R}$ is the reinforcement rule for the classical P\'olya urn scheme. Now using  
 ~\eqref{Equ:Basic-Recursion} we have
\[
\bE\left[U_{n+1} \,\Big\vert\, \mathcal{F}_n\right] = U_n + \frac{U_n}{n+1} =(n+2) \frac{U_n}{n+1},
\]
which implies that each coordinate of the vector $\frac{U_n}{n+1}$, is a positive martingale and hence  converges. 
Moreover, by exchangeability and arguments similar to the classical P\'olya urn, we can easily show that,
\[ 
\frac{U_{n,0}}{n+1} \longrightarrow Z \mbox{\ \ a.s.},
\]  
where $Z \sim Beta(U_{0,0}, U_{0,1})$. 
 \end{remark}

\subsection{Asymptotics of the Colour Count Statistics}
\label{SubSec:Colour-Count-Results}

\subsubsection{Case when $\hat{R}$ is Irreducible.}
\label{SubSubSec:Colour-Count-Results-Irreducible}
\begin{theorem} 
\label{Thm:colour-count1}
Suppose $\hat{R}$ is irreducible then,
\[\frac{N_{n,j}}{n} \longrightarrow \frac{1}{k\theta-1}\left[\theta -\mu_j\right], \;\;\;a.s.\;\;\forall\; 0 \leq j\leq k-1,\] 
where $\mu=(\mu_0,\mu_1,\hdots, \mu_{k-1})$ satisfies equation ~\eqref{Equ:Def-mu}.
\end{theorem}

\begin{theorem}\label{Thm:colour-count1-clt}
Suppose $\hat{R}$ is irreducible, then there exists a \emph{variance-co-variance} matrix 
$\tilde \Sigma \equiv {\tilde \Sigma}\left(\theta, k\right)$, such that
\[ \frac{N_n - \frac{n}{k\theta -1}(\theta \bone-\mu) }{\sigma_n} \implies  {\mathcal N}_k \left(0, \tilde \Sigma \right), \] 
where $\sigma_n$ is given in the equations ~\eqref{Equ:Def-sigma-n-k} and ~\eqref{Equ:Def-sigma-n-2}. Moreover, 
\begin{equation}
\Sigma = R^T\tilde \Sigma R,
\label{Equ:Relation-between-Sigma-and-Singma-tilde}
\end{equation}
where $\Sigma$ is as in Theorem \ref{Thm-Un-CLT}.
\end{theorem} 

\begin{remark}
Note that from definition ~\eqref{Def:Nn}, it follows that $\sum_{j=0}^{k-1} N_{n,j} = n$, thus
$\tilde \Sigma$ is a \emph{positive semi-definite} matrix. Further, from equation ~\eqref{Equ:Relation-between-Sigma-and-Singma-tilde} 
it follows that $\mbox{rank}\left(\Sigma\right) \leq \mbox{rank}\left(\tilde \Sigma\right)$ and equality holds, if and only if,
the replacement matrix $R$ is non-singular. 
\end{remark}

\subsubsection{Case when $\hat{R}$ is Reducible.}
\label{SubSubSec:Colour-Count-Results-Reducible}
Recall that $\hat{R}$ has the form given in ~~\eqref{R-Star} when it is reducible.
\begin{theorem} \label{Thm:colour-count2}
Let $R$ be a star matrix with $0$ as a central vertex and $\theta =1$, such that $R \neq \begin{bmatrix} 0&1\\1&0\end{bmatrix}$, then
\[\frac{N_{n,0}}{n} \longrightarrow 0, \;\;\;a.s.\] 	
and,
\[\frac{N_{n,j}}{n} \longrightarrow \frac{1}{k-1}, \;\;\;a.s.\; \forall 1\leq j\leq k-1.\] 	
\end{theorem}

\begin{remark}
 For $R = \begin{bmatrix} 0&1\\ 1&0\end{bmatrix}$ using equation ~\eqref{Un-Nn} and  Remark ~\eqref{Friedman-2color-Un} we get
 \[ 
 \frac{N_{n,0}}{n+1} \longrightarrow 1-Z \mbox{\ \ a.s.},
 \]
 where as before, $Z \sim Beta(U_{0,0}, U_{0,1})$. 
\end{remark}

\begin{theorem} \label{Thm:colour-count2-clt}
Let $R$ be a star matrix  with $0$ as a central vertex and $\theta =1$,
such that $R\neq \begin{bmatrix} 0&1\\ 1&0\end{bmatrix}$, then\\
\begin{enumerate}
 \item if $\gamma =\frac{1-\alpha_0}{k-1} < 1/2$ , then
\[\frac{1}{\sqrt n}\left( N_{n,-} - \frac{n}{k-1}\bone \right)\implies {\mathcal N}_k\left(\mathbf{0},\frac{1}{k-1}I- \frac{1}{(k-1)^2}J\right),\] 	
where $N_{n,-} = (N_{n,1}, \cdots, N_{n,k-1})$, and 
\[\frac{N_{n,0}}{\sqrt n} \cp 0.\]
\item if $\gamma =\frac{1-\alpha_0}{k-1} \geq  1/2$ , then
\[\frac{1}{n^\gamma}\left( N_{n,j} - \frac{n}{k-1}\right)\cp \frac{\alpha_j}{k-1} W, \;\;\;\;\forall j \neq 0\] 	
and
\[\frac{N_{n,0}}{n^\gamma} \cp \frac{1-\alpha_0}{k-1}W.\]
where $W$ is as given in Theorem \ref{Thm-star-1}.
\end{enumerate}
\end{theorem}

\begin{remark}
Note that $\gamma < 1/2$, if and only if, $k \geq 4$ or $k = 3$ and $\alpha_0 >0$ or $k=2$ and $\alpha_0 > 1/2$. 
\end{remark}

\section{Proof of the Necessary and Sufficient Condition for $\hat{R}$ to be Irreducible}
\label{Sec:Proof-of-Necessary-Sufficient-Condition-on-hatR}


Suppose $G$ and $\hat{G}$ are the directed graphs associated with the matrices $R$ and $\hat{R}$ respectively, 
as defined earlier. Observe that, $\hat{R}$ is the product of two stochastic matrices, $R$ and $A$. The underlying Markov chain of $\hat{R}$
can be seen as a two step Markov chain where the first step is taken  according to $R$ and the second step is taken according to $A$.
Recall from equation ~\eqref{Def:R-hat} that
\[ 
\hat{R} =  \frac{1}{k\theta-1}\left(\theta J-R\right).
\]
Now, to show that the Markov chain associated with $\hat{R}$ is irreducible, it is enough to show that there exist a directed path between any 
two fixed vertices say $u$ and $v$, in $\hat{G}$.

Clearly for $\theta > 1$, $\hat{R}_{uv} >0$ for all $u,v$, and thus $\hat{R}$ is irreducible.
Therefore, we only have to verify irreducibility for $\theta= 1$ case. For this we first fix two vertices, say $u$ and $v$. From 
equation ~\eqref{Def:R-hat} we get
\begin{equation}\label{Rhat-definition}
\hat{R}_{uv}=  \frac{1 - R_{u,v}}{k -1}.
\end{equation}

To complete the proof, we will show that there is a path from $u$ to $v$ of length at most $2$.
We consider the following two cases:

{\bf Case 1 $R_{u,v} < 1$:} In this case, from equation ~\eqref{Rhat-definition} we get, $\hat{R}_{uv} >0$.
Therefore $(u,v)$ is an edge in $\hat{G}$ and trivially there is a path of length $1$ from $u $ to $v$ in $\hat{G}$. 

{\bf Case 2 $ R_{u,v} =1$:} 
In this case, $u$ has no $R$-neighbor other than $v$, that is $(u,v)$ is the only incoming edge to $v$ in $G$
and from equation ~\eqref{Rhat-definition}, we have 
\[\hat{R}_{uv}= 0.\]

As mentioned earlier for $\theta = 1$ and $k=2$, $\hat{R}$ is reducible only when $R$ 
is the Friedman urn scheme, which is a star with two vertices. Thus in the rest of 
the proof we take $k>2$, and  show that $\hat{R}^2_{uv} >0$, that is there is a path
of length $2$. 
%

Now, if $R$ is not a star then there must exists a vertex $l$ such that it leads to a vertex other than the central vertex,
say $m$  that is $R_{l,m}>0$ ($m \neq v$). 
Now, according to $\hat{R}$ chain, there is a positive probability of going from $u$ to $l$ in one step
(first take a R-step from $u$ to $v$ which happens with probability $1$ is this case, as $R_{u,v}=1$, 
and then take a A-step to $l$ with probability $1/(k-1)$) and a positive probability
of going from $l$ to $v$ in one step (first take a R-step from $l$ to $m$ with probability $R_{l,m}$, and then 
take a A-step to $v$ with probability $1/(k-1)$).
Therefore, there is path of length two in $\hat{G}$ from $u$ to $v$ and thus the chain is irreducible.


\begin{remark}
Note that from the proof it follows that for a replacement matrix $R$ with $k > 2$ 
such that, $\hat{R}$ is irreducible, then $\hat{R}$ is also aperiodic. 
\end{remark}

\section{Proofs of the Main Results}
\label{Sec:Proofs}

We begin by observing the following fact. From equations ~~\eqref{Equ:Basic-Recursion}, ~~\eqref{prob}, ~~\eqref{Equ:Def-A}  we get,
\begin{equation}
\bE\left[U_{n+1}\,\Big\vert\,\mathcal{F}_n\right] = U_n+ \bE\left[\chi_{n+1}\,\Big\vert\,\mathcal{F}_n\right]R = U_n+ \frac{U_n}{n+1} AR.  
\label{Equ:Mean-negatively reinforced}
\end{equation}
Thus, 
\begin{equation} \label{Equ:Cond-Mean}
\bE\left[U_{n+1}A\,\Big\vert\,\mathcal{F}_n\right] =U_nA+ \frac{U_n}{n+1} ARA.
\end{equation}
Let $\hat{U}_n := U_n A, n \geq 0$, then 
\begin{equation}
\hat{U}_{n+1} = \hat{U}_n + \chi_{n+1} RA.
\label{Equ:Hat-Basic-Recursion} 
\end{equation}
and from equation ~\eqref{Equ:Cond-Mean} we get 
\[ \bE\left[\hat{U}_{n+1}\,\Big\vert\,\mathcal{F}_n\right]= \hat{U}_n+ \frac{\hat{U}_n}{n+1} RA. \]
Therefore $\left(\hat{U}_n\right)_{n \geq 0}$ is a classical urn scheme (uniform selection), with replacement matrix $RA$.
The construction $\left(\hat U_n\right)_{n \geq 0}$ is essentially a coupling of a negative but linearly reinforced urn 
$\left(U_n\right)_{n \geq 0}$ with replacement matrix $R$, to a classical (positively reinforced) urn $(\hat{U}_n)_{n \geq 0}$
with replacement matrix $\hat{R}$. Note that, we get a one to one correspondence, as $A$ is always invertible. 

\begin{proof} [Proof of Theorem \ref{Thm-Un-Strong-convergence}]
Recall that, $\hat{U}_n =U_nA$ is the configuration of a classical urn model 
with replacement matrix $\hat{R}$. Since by our assumption, $\hat R$ is irreducible therefore by
Theorem 2.2. of \cite{BaiHu05},  the limit of $\frac{1}{n+1}\hat{U}_n$ is the
normalized left eigenvector of $\hat{R}$ associated with the maximal eigenvalue $1$. That is 
\[ \frac{\hat{U}_n}{n+1} \longrightarrow \nu, \; a.s.\]
where $\nu$ satisfies
\[\nu \hat{R} = \nu.\]
Since $U_n  = \hat{U}_n A^{-1}$, we have

\[\frac{U_n}{n+1} \longrightarrow \mu, \;\; a.s., \]  
where  $\mu = \nu A^{-1}$,  and it satisfies the following matrix equation:
\[\left(\theta \bone - \mu\right) R = \left(k\theta -1\right) \mu.\]
This completes the proof.
\end{proof}

\begin{proof}[Proof of Theorem \ref{Thm-Un-CLT}]
Let $1, \lambda_1, \hdots \lambda_s$ be the distinct eigenvalues of $R$, such that, 
$1\geq \Re(\lambda_1)\geq \cdots \geq \Re(\lambda_s)\geq -1$, where $\Re(\lambda)$ denotes the real part of the eigenvalue $\lambda.$
Recall from equation ~\eqref{Def:R-hat} that $\hat{R} = \frac{1}{k \theta -1} \left(\theta J_k - R\right)$.
So the eigenvalues of $\hat R$ are $1, b \lambda_1, \cdots ,b \lambda_s$, where $b  = \frac{-1}{k\theta -1}$. 
Let  $\tau = \max\{0, b \, \Re( \lambda_s )\}$.
Since $\hat{U}_n = U_nA$, is a classical urn scheme with replacement matrix $\hat{R}$, using Theorem 3.2 of \cite{BaiHu05}, 
if 
\begin{equation}
b \, \Re(\lambda_s) \leq \frac{1}{2}
\end{equation}
then there exists a variance-co-variance matrix $\Sigma'$, such that 
\[\frac{\hat{U}_n-n\nu}{\sigma_n} \implies{\mathcal N}_k\left(0,\Sigma'\right)\]
where 
\begin{equation}
\sigma_n  =\begin{cases}
            \sqrt{n \log n} & \text{if}  \;\; b \, \Re\left(\lambda_s\right) = \frac{1}{2},\\
                            & \\
            \sqrt{n}        & \text{if}  \;\; b \, \Re\left(\lambda_s\right) < \frac{1}{2}.
	  \end{cases}
\end{equation}
Notice that, 
\begin{equation} \label{max-eigenvalue-condition}
b \, \Re(\lambda_s) \leq \frac{1}{2} \iff \Re(\lambda_s)\geq - \frac{1}{2}(k \theta -1).
\end{equation}
Now since $\theta \geq 1$ and $\Re(\lambda_s)\geq -1$ the above equation ~\eqref{max-eigenvalue-condition} holds whenever $k \geq 3$.
Further, for $k \geq 3$, equality in ~\eqref{max-eigenvalue-condition} holds if and only if, $\theta = 1$, and $k =3$. Moreover,
for $k=2$, the condition is equivalent to $\Re\left(\lambda_s\right) \leq \frac{1-2\theta}{2}$. Thus, 
$\sigma_n$ is given in ~\eqref{Equ:Def-sigma-n-k} and ~\eqref{Equ:Def-sigma-n-2}
Therefore, 
\[\frac{U_n-n\mu}{\sigma_n} \implies{\mathcal N}_k\left(0,\Sigma\right)\]
where $\Sigma = A^T\Sigma' A$.
\end{proof}

\begin{proof} [Proof of Theorem  \ref{Thm-star-1}]
Without lose of any generality, we will assume $\gamma > 0$ (equivalently $\alpha_0 < 1$), as otherwise the result is trivial as described in 
Remark \ref{Rem:Gamma-0-Case}. 
Since the matrix $\hat{R}$, as given in  ~\eqref{R-hat-star} is reducible without 
isolated blocks. Using Proposition 4.3 of \cite{Gouet} we get,
\[\frac{\hat{U}_{n,0}}{n+1} \to 0 \;\;\text{and} \;\; \frac{\hat{U}_{n,j}}{n+1} \to \frac{1}{k-1}, \; \forall j \neq 0. \]
which implies 
\[\frac{U_{n,0}}{n+1} \to 1 \;\;\text{and} \;\; \frac{U_{n,j}}{n+1} \to 0, \; \forall j \neq 0. \]
Now, note that the matrix $\hat R$ given in ~\eqref{R-hat-star} has eigenvalues $1, \gamma$ and 
$0,0,\dots, 0$ ($k-2$ times), where $\gamma = (1-\alpha_0)/(k-1)$. 
The eigenvector corresponding to the non-principal eigenvalue $\gamma$ is 
\[ \xi = \frac{1}{\gamma} \begin{pmatrix}0,1,1, \hdots, 1\end{pmatrix}'.\]
Therefore,
\[\bE\left[U_{n+1}\xi\,\Big\vert\,\mathcal{F}_n\right] = U_n\left[I+\frac{\hat R}{n+1}\right]\xi  = U_n\xi \left[1+\frac{\gamma}{(n+1)}\right] .\]
Let $ \Pi_n\left(\gamma\right) = \prod_{i=1}^n \left[1+\frac{\gamma}{i}\right]$
then, $W_n \coloneqq U_n \xi/ \Pi_n\left(\gamma\right)$ is a non-negative martingale and using Euler's product, for large $n$ 
\[\Pi_n\left(\gamma\right) \sim \frac{n^{\gamma}}{\Gamma (\gamma+1)}.\]
We now show that this martingale is $\mathcal{L}^2$ bounded, which will then imply that

\begin{equation}
\frac{U_n \xi}{n^\gamma} \longrightarrow W
\label{Equ-star-convg}
\end{equation}
where $W$ is a non-degenerate random variable. More precisely, $W$  is nonzero with positive probability. We can write,

\[ \bE\left[W_{n+1}^2\,\Big\vert\,\mathcal{F}_n\right]= W_n^2+ \bE\left[(W_{n+1}-W_n)^2\,\Big\vert\,\mathcal{F}_n\right] \]
and
\begin{align*}
W_{n+1}- W_n &=   \frac{1}{\Pi_{n+1}(\gamma)}\left[U_{n+1}\xi - U_n\xi \left(1+\frac{\gamma}{n+1}\right) \right]\\
&= \frac{1}{\Pi_{n+1}(\gamma)}\left[\chi_nR\xi - \frac{\gamma}{n+1}U_n\xi \right]\\
& = \frac{1}{\Pi_{n+1}(\gamma)} \left[ (k-1)\chi_{n,0} - \frac{(n+1) - U_{n,0}}{n+1} \right]\\
& = \frac{k-1}{\Pi_{n+1}(\gamma)} \left[\chi_{n,0} - \bE\left[\chi_{n,0}\,\Big\vert\,\mathcal{F}_n\right] \right]\\
\end{align*}

Therefore,
\begin{align*}
 \bE\left[W_{n+1}^2\,\Big\vert\,\mathcal{F}_n\right]&  W_n^2+ \frac{(k-1)^2}{\Pi^2_{n+1}(\gamma)} 
 \left( \bE\left[\chi_{n,0}\,\Big\vert\,\mathcal{F}_n\right] - \bE\left[\chi_{n,0}\,\Big\vert\,\mathcal{F}_n\,\right]^2\right)\\
 & \leq W_n^2+ \frac{(k-1)^2}{\Pi^2_{n+1}(\gamma)} \bE\left[\chi_{n,0}\,\Big\vert\,\mathcal{F}_n \,\right] \\
 & = W_n^2+ \frac{1-\alpha_0}{(n+1)\Pi_{n+1}(\gamma)} \frac{U_n\xi}{ \Pi_{n+1}(\gamma)} \\
 & \leq W_n^2+ \frac{1-\alpha_0}{(n+1)\Pi_{n+1}(\gamma)} W_n  \\
 & \leq W_n^2+ \frac{(1-\alpha_0) \Gamma(\gamma+1)}{2(n+1)^{\gamma+1}}(1+W^2_n)  
 \end{align*}
The last inequality holds because $2W_n \leq1+W_n^2$.
Let $c \coloneqq \frac{1}{2}(1-\alpha_0)\Gamma (\gamma+1)$, then 
 
 \begin{align*}
 \bE\left[W_{n+1}^2+1\,\Big\vert\,\mathcal{F}_n\right]&\leq\left(1+\frac{c}{(n+1)^{\gamma+1}}\right) \left(1+W^2_n\right)\\
 & \leq (1+W^2_0) \prod_{j=1}^n \left(1+\frac{c}{(j+1)^{\gamma+1}}\right)\\
 & \leq (1+W^2_0)\exp \left( \sum_{j=1}^n \frac{c}{(j+1)^{\gamma+1}}\right) <\infty \;\;\;\; (\text { since } \gamma >0).
 \end{align*}
Thus $W_n$ is $\mathcal{L}^2$-bounded and hence converges to a non-degenerate random variable say $W$.
Now for a star matrix $R$ (as given in equation ~\eqref{R-Star}), the recursion ~\eqref{Equ:Basic-Recursion} reduces to

\begin{equation*}
U_{n+1,0}  = U_{n,0} + \alpha_0 \chi_{n+1,0} + (1-\chi_{n+1,0}).
\end{equation*}
and 
\begin{equation}
U_{n+1,h}  = U_{n,h} + \alpha_h \chi_{n+1,0}\;\;\; \forall h \neq 0
\label{Equ:Star-Recursion}  
\end{equation}

Recall that for $ h \neq 0$, $\alpha_h > 0$, dividing both sides by $\alpha_h$, we get 
 \[ \frac{U_{n+1,h}}{\alpha_h} =  \frac{U_{0,h}}{\alpha_h}+ \sum_{j=1}^{n+1} \chi_j.\]
Since the above relation holds for every choice of $h>0$, we get

\begin{equation}
\frac{U_{n+1,h}}{\alpha_h}-\frac{U_{n+1,l}}{\alpha_l} = \frac{U_{0,h}}{\alpha_h}-\frac{U_{0,l}}{\alpha_l}
\label{Equ-compare-convg}
\end{equation}
for any $h,l \in \{ 1,2,\cdots,k-1\}$. Multiplying the above equation by
$\frac{\alpha_l}{1-\alpha_0}$ and taking sum over $l\neq 0$, we get
\[  \frac{U_{n,h}}{\alpha_h} - \frac{1}{1-\alpha_0} \sum_{l\neq 0} U_{n,l} =
\frac{U_{0,h}}{\alpha_h} - \frac{1}{1-\alpha_0}\sum_{l\neq 0} U_{0,l},  \]
which can be written as,
\[\frac{U_{n,h}}{\alpha_h} - \frac{1}{k-1}U_n \xi = \frac{U_{0,h}}{\alpha_h} - \frac{1}{k-1}U_0\xi. \]
Now dividing both sides by $n^{\gamma}$,
\[ \frac{1}{n^{\gamma}}\frac{U_{n,h}}{\alpha_h} - \frac{1}{k-1}\frac{U_n \xi}{n^{\gamma}} = \frac{1}{n^{\gamma}}\left[\frac{U_{0,h}}{\alpha_h} - \frac{1}{k-1}U_0\xi \right].\]
Note that the right hand side of the above expression goes to $0$ as $n$ tends to infinity. Therefore 
\begin{align*}
\lim_{n \to \infty} \frac{1}{n^{\gamma}}\frac{U_{n,h}}{\alpha_h} - \frac{1}{k-1}\frac{U_n \xi}{n^{\gamma}} =0
 \end{align*}
Using the limit from ~\eqref{Equ-star-convg} we get,
\[\frac{U_{n,h}}{n^{\gamma}} \longrightarrow \frac{\alpha_h}{k-1} W.\]
\end{proof}

\begin{proof} [Proof of Theorem \ref{Thm:colour-count1}] 
Note that from equation ~\eqref{Def:Nn} and ~~\eqref{Un-Nn}, we can write
\begin{align}
 N_n &= \sum_{i=1}^{n}\left( \chi_i- \bE\left[\chi_i \,\Big\vert\,\mathcal{F}_{i-1}\right] \right) + \sum_{i=1}^{n} \bE\left[\chi_i \,\Big\vert\,\mathcal{F}_{i-1}\right]\nonumber \\
 & = \sum_{i=1}^{n}\left( \chi_i- \bE\left[\chi_i \,\Big\vert\,\mathcal{F}_{i-1}\right] \right)+ \frac{1}{k\theta-1} \left[ \theta \bone  - \frac{U_{i-1}}{i}\right],
\label{Nn-sum}
\end{align}
Since $\left(\chi_i- \bE\left[\chi_i \,\Big\vert\,\mathcal{F}_{i-1}\right]\right)_{i\geq1}$ 
is a bounded martingale difference sequence, using Azuma's inequality (see \cite{Bi95}) we get

\begin{equation}
\frac{1}{n}\sum_{i=1}^{n} \left( \chi_i- \bE\left[\chi_i \,\Big\vert\,\mathcal{F}_{i-1}\right]\right) \longrightarrow 0, \;\; a.s.. 
\label{Nn1-Martingale-limit}
\end{equation}
Using Theorem \ref{Thm-Un-Strong-convergence} and \emph{Cesaro Lemma} (see \cite{Apostol}), we get
\[\frac{N_{n,j}}{n} \longrightarrow \frac{1}{k\theta-1}\left[\theta -\mu_j\right], \;\;\;a.s.\;\;\forall\; 0 \leq j\leq k-1.\] 
\end{proof}

\begin{proof}[Proof of Theorem \ref{Thm:colour-count1-clt}]
Notice that under our coupling 
$N_n$ remains same for the two processes, namely,  $\left(U_n\right)_{n \geq 0}$ and $\left(\hat{U}_n\right)_{n \geq 0}$. Thus applying 
Theorem 4.1 of \cite{BaiHu05} on the urn process $\left(\hat{U}_n\right)_{n \geq 0}$ we conclude that
there exists a matrix $\tilde \Sigma$ such that,
\begin{align*}
\frac{N_n - n \mu A}{\sigma_n}&\implies  {\mathcal N}_k\left(0,\tilde\Sigma\right)
\end{align*}
Finally the equation ~\eqref{Equ:Relation-between-Sigma-and-Singma-tilde} follows from ~\eqref{Un-Nn}.
This completes the proof.
\end{proof}

\begin{proof} [Proof of Theorem \ref{Thm:colour-count2}] 
The proof follows from equation ~\eqref{Nn-sum} and ~\eqref{Nn1-Martingale-limit}.
\end{proof}

\begin{proof}[Proof of Theorem \ref{Thm:colour-count2-clt}]
Let $M_n= (M_{n,0},M_{n,1}, \cdots, M_{n,k-1}) $ be a martingale, 
where $M_{n,j} \coloneqq \sum_{i=1}^n \chi_{i,j}- \bE\left[\chi_{i,j}\,\Big\vert\,\mathcal{F}_{i-1}\right]$. Note that $(M_n)$ is a a bounded increment martingale.
and $X_n \coloneqq \frac{1}{\sqrt{n}} M_n$.
That is, for a fixed colour $j$, $X_{n,j} = \frac{1}{\sqrt{n}}\left(\chi_{i,j}- \bE\left[\chi_{i,j}\,\Big\vert\,\mathcal{F}_{i-1}\right]
\right)$. 
Let $M_{n,-} \coloneqq (M_{n,1}, \cdots, M_{n,k-1})$  and $X_{n,-} \coloneqq (X_{n,1}, \cdots, X_{n,k-1})$.\\
In this proof, we first provide a central limit theorem for $M_{n,-}$, and then for $N_n$. 
Observe that the $(l,m)$-th entry of the matrix $\bE\left[X_{i,-}^TX_{i,-}\,\Big\vert\,\mathcal{F}_{i-1}\right]$ is 
\begin{eqnarray*} 
 &  & 
    \frac{1}{n} \bE\left[\chi_{i,l}\chi_{i,m}\,\Big\vert\,\mathcal{F}_{i-1}\right] 
    -\bE\left[\chi_{i,l}\,\Big\vert\,\mathcal{F}_{i-1}\right] \bE\left[\chi_{i,m}\,\Big\vert\,\mathcal{F}_{i-1}\right]\\
 & = &
 \begin{cases}
 \frac{1}{n} \bE\left[\chi_{i,l}\,\Big\vert\,\mathcal{F}_{i-1}\right]\left(1 -\bE\left[\chi_{i,l}\,\Big\vert\,\mathcal{F}_{i-1}\right]\right) &\text{if } l=m,\\
 \frac{-1}{n}\bE\left[\chi_{i,l}\,\Big\vert\,\mathcal{F}_{i-1}\right] \bE\left[\chi_{i,m}\,\Big\vert\,\mathcal{F}_{i-1}\right] &\text{if } l\neq m
 \end{cases} \\
 & = &
 \begin{cases}
 \frac{1}{n(k-1)}\left(1-\frac{U_{i-1,l}}{i}\right)\left(1 - \frac{1}{k-1}\left(1-\frac{U_{i-1,j}}{i}\right)\right) &\text{if } l=m,\\
 \frac{-1}{n(k-1)^2}\left(1-\frac{U_{i-1,l}}{i}\right)\left(1-\frac{U_{i-1,m}}{i}\right) &\text{if } l\neq m,
 \end{cases} \\
 \end{eqnarray*}
So, as $n \to \infty$, (using Theorem \ref{Thm-star-1}) we have

\begin{align*} \label{limit-var-}
 \sum_{i=1}^n \bE\left[X_{i,-}^TX_{i,-}\,\Big\vert\,\mathcal{F}_{i-1}\right]_{(l,m)}& \to 
 \begin{cases}
\frac{(k-2)}{(k-1)^2}  &\text{if } l=m,\\
\frac{-1}{(k-1)^2} &\text{if } l\neq m,\\
\end{cases}
\end{align*}
Therefore,
\[ \sum_{i=1}^n \bE\left[X_{i,-}^TX_{i,-}\,\Big\vert\,\mathcal{F}_{i-1}\right] \to \frac{1}{k-1} I - \frac{1}{(k-1)^2} J,\]
and by the martingale central limit theorem \cite{HallHeyde} , we get

\begin{equation}\label{Mn-CLT}
\frac{1}{\sqrt{n}} M_{n,-} \implies {\mathcal N}_k\left(0,\frac{1}{k-1} I - \frac{1}{(k-1)^2} J\right) 
\end{equation}
Now for colour $0$, we have 
\begin{align*}
\frac{1}{\sqrt{n}}M_{n,0} = \frac{1}{\sqrt{n}} \sum_{j=1}^{k-1}M_{n,-}
\end{align*}
which implies 
\[\frac{1}{\sqrt{n}}M_{n,0} \cp 0.\]
We now prove the central limit theorem for $N_n$. By equation ~\eqref{Nn-sum}, we have 
\[N_{n} =  M_{n} + \sum_{i=1}^{n} \bE\left[\chi_{i} \,\Big\vert\,\mathcal{F}_{i-1}\right]\]
Therefore,
\begin{equation}\label{Nn-Mn}
 N_{n,-} -\frac{n}{k-1}\bone  =  M_{n,-} - \frac{1}{k-1}\sum_{i=1}^{n}\frac{U_{i-1,-}}{i}
\end{equation}
Form Theorem \ref{Thm-star-1}, we know that for each $j\neq 0$
\[\frac{U_{i-1,j}}{i^\gamma} \to \frac{\alpha_j}{k-1}W, \; a.s.. \]
\begin{align*}
\sum_{i=1}^{n}\frac{U_{i-1,j}}{i} &\asymp \frac{\alpha_j}{k-1} W \sum_{i=1}^n i^{\gamma-1}\\
&\sim  \frac{\alpha_j}{k-1} W  n^\gamma.
\end{align*}
Therefore,
\begin{equation}
 \frac{1}{n^\gamma}\sum_{i=1}^{n}\frac{U_{i-1,j}}{i} \to  \frac{\alpha_j}{k-1} W \,a.s.. 
 \label{Nn-2nd-term-limit}
\end{equation}
Therefore for $\gamma < 1/2$, using equation ~\eqref{Mn-CLT}, ~\eqref{Nn-Mn} and ~\eqref{Nn-2nd-term-limit} we get

\[\frac{1}{\sqrt n}\left( N_{n,-} - \frac{n}{k-1}\bone \right)\implies {\mathcal N}_k\left(0, \frac{1}{k-1} I - \frac{1}{(k-1)^2} J\right),\]
and for $\gamma \geq 1/2$,

\[\frac{1}{n^\gamma}\left( \frac{n}{k-1}- N_{n,j} \right)\cp \frac{\alpha_j}{k-1}W \;\forall j,\] 	
since then $M_{n,j}/n^{\gamma} \cp \mathbf{0}$. 
For $j=0$, we have
  
\[N_{n,0} =  n- \sum_{j=1}^{k-1}N_{n,j} = \sum_{j=1}^{k-1}\left(\frac{n}{k-1} -N_{n,j}\right)\]

Therefore for $\gamma \geq 1/2$, we have 
\[\frac{1}{n^\gamma}N_{n,0} = \sum_{j=1}^{k-1}\frac{1}{n^\gamma}\left(\frac{n}{k-1} -N_{n,j}\right) \cp \frac{1}{k-1}W\sum_{j=1}^{k-1}\alpha_j = \frac{1-\alpha_0}{k-1}W.\]
and for $\gamma> 1/2$ we have
\[\frac{N_{n,0}}{\sqrt n } \cp 0.\]

\end{proof}

\bibliographystyle{plain}
\bibliography{ref1}


\end{document}